\documentclass[10pt, reqno]{amsart}
\usepackage{tikz}
\usepackage{amssymb,latexsym}
\usepackage{eucal}

\numberwithin{equation}{section}
\newtheorem*{Prb}{Problem 1.3}

\newtheorem{thm}{Theorem}[section]

\newtheorem{lem}[thm]{Lemma}

\newcommand{\p}{\partial}
\newcommand{\e}{\varepsilon}

\newcommand{\ph}{\varphi}

\newcommand{\A}{\mathcal{A}}

\newcommand{\D}{\Delta}

\newcommand{\AL}{\mathrm{A}}
\newcommand{\BT}{\mathrm{B}}

\newcommand{\C}{\mathcal{C}}

\newcommand{\GG}{\mathcal{G}}
\newcommand{\FF}{\mathcal{F}}

\newcommand{\X}{\mathcal{X}}
\newcommand{\Y}{\mathcal{Y}}

\newcommand{\wtl}{\widetilde}
\newcommand{\wht}{\widehat}

\newcommand{\crd}{cyclically reduced}
\newcommand{\er}{\eqref}
\newcommand{\cntsd}{contiguity subdiagram}

\newcommand{\rgnb}{regular neighborhood}
\newcommand{\cncm}{connected component}
\newcommand{\ifff}{if and only if}

\begin{document}
\title
[Embedding of groups and quadratic equations over groups]
{Embedding of  groups and quadratic equations over groups}
 \author{D. F. Cummins }
 \address{  Department of Mathematics\\
 United States Military Academy \\
  West Point \\ NY  10996 \\ U.S.A. } \email{desmond.cummins@usma.edu}
 \author{S. V. Ivanov }
 \address{  Department of Mathematics\\
 University of Illinois \\
  Urbana\\  IL 61801\\ U.S.A. } \email{ivanov@illinois.edu}

%Urladdr{}
\thanks{The second author was supported in part by the NSF under grant  DMS 09-01782}
%\date{}

\begin{abstract}
We prove that, for every integer $n \ge 2$,   a finite or infinite countable group $G$ can be embedded into a 2-generated group $H$ in such a way that the solvability of quadratic equations of  length at most $n$
is preserved, i.e.,  every quadratic equation  over $G$  of  length at most $n$  has a solution in $G$  if and only if this equation, considered as an equation  over $H$, has a solution in  $H$.
\end{abstract}

\maketitle

\section{Introduction}

It  is  a classical  result of Higman, B. Neumann, and H. Neumann \cite{HNN} that every finite or infinite countable group can be embedded into a 2-generated group. In this note, we are concerned with such an emdedding that would preserve the solvability of every quadratic equation of bounded length.

We start with definitions. Let $G$ be a finite or infinite countable group  and let
 \begin{equation}\label{e0}
    G =  \langle \, a_1, a_2, \dots  \, \| \, R_1=1, R_2=1, \dots  \, \rangle
\end{equation}
be a presentation for $G$ by means of generators $a_1, a_2,  \dots$ and defining relations
$R_1=1, R_2=1,  \dots$, where $R_1, R_2,  \dots$  are nonempty \crd\ words  over the alphabet $\A^{\pm 1} := \{  a_1^{\pm 1} , a_2^{\pm 1},  \dots \}$. If $U$ is a word over $\A^{\pm 1}$ and the  image of $U$ in $G$ is trivial, we write $U \overset G  = 1$ or say that $U = 1$ in $G$.

Let $\X$ be a finite or infinite countable set, called a set of variables,  $\X^{-1} := \{ x^{-1} \mid x \in \X \}$, and $\X^{\pm 1} := \X \cup \X^{-1}$.  Let
$\FF(\X)$ denote the free group with the free base $\X$ and let
$G * \FF(\X)$ denote the free product of $G$ and $\FF(\X)$.  Elements of $G * \FF(\X)$  can  be regarded as words over the alphabet $\Y^{\pm 1}$,  where   $\Y := \A   \cup \X$.

A word  $W = y_1 \dots y_\ell$ over $\Y^{\pm 1}$, where $y_1, \dots, y_\ell \in \Y^{\pm 1}$, is called {\em reduced} if $\ell >0$, i.e., $W$ is not empty,  and  $W $ contains no subwords of the form $y y^{-1}$  or  $y^{-1} y$, where $y \in \Y$.  A word  $W$  over $\Y^{\pm 1}$ is {\em \crd\ }  if $W$ is reduced and every cyclic permutation of $W$ is reduced.
The {\em length} of a word    $W = y_1 \dots y_\ell$    over $\Y^{\pm 1}$  is $\ell = | W|$ and the {\em $\X$-length}   $| W|_\X$ of $W$ is the number of all occurrences of letters of $\X^{\pm 1}$  in the word $W$.
For example, $| a_1 x_1 x_2 a_2^{-1} x_1^{-1} |_\X =3$ if  $a_1, a_2 \in \A$ and $x_1, x_2 \in \X$.

An {\em equation} over $G$ is a formal expression $W = 1$, where $W$ is a \crd\ word over $\Y^{\pm 1}$ with  $| W|_\X >0$.  The  {\em length}  of an equation   $W = 1$ over $G$ is the number  $| W|_\X$.
The  {\em total length}  of an equation   $W = 1$ over $G$ is  $| W|$.
An equation   $W = 1$   over $G$ is called {\em quadratic} if, for every letter $x \in \X$,  the sum of the number of occurrences of $x$ in $W$ and the number of occurrences of $x^{-1}$  in $W$ is either 2 or 0.

We say that an equation   $W = 1$   over $G$ has a {\em solution}    if there exists a
homomorphism $\psi_W : G*\FF(\X) \to G$ which is identical on $G$ and which takes the word $W \in G*\FF(\X)$ to the identity, i.e., $\psi_W |_{G} = \mbox{id}_G$  and  $\psi_W(W) = 1$ in $G$. Let $x_1, \ldots, x_k$ be all letters of $\X$ that occur  in $W$ or in $W^{-1}$. A {\em solution tuple } to the equation $W=1$, defined by a homomorphism $\psi_W : G*\FF(\X) \to G$, is a tuple $(U_1, \dots, U_k)$, where $U_1, \dots, U_k$ are some words over $\A^{\pm 1}$, such that
$\psi_W(x_j) = U_j$ in  $G$ for every $j=1, \dots, k$. The {\em length} of a solution tuple $(U_1, \ldots, U_k)$   to the equation $W=1$ is the sum  $\sum_{j=1}^k |U_j|$.

If $\mu: G\to H$ is a group monomorphism and $W=1$ is an equation over $G$, then we can use $\mu$ and $W=1$ to obtain an equation over $H$ by replacing every letter $a_i^{\e}\in\A^{\pm 1}$, $\e = \pm 1$, that appears in $W=1$ with $\mu(a_i^{\e})$.  This new equation over $H$  is denoted by $\mu(W)=1$.

\begin{thm}\label{thm1} Let $n \ge 2$ be an integer and let $G$ be a finite or infinite countable  group.  Then there exists  an embedding $\mu_n : G \to H$ of $G$ into a 2-generated  group $H = \langle h_1, h_2 \rangle$,  that preserves the solvability of every quadratic equation $W =1$ over $G$ of length  $| W|_\X \le n $, i.e., for every
equation  $W =1$ over $G$  of length at most $n$,  the  equation $W =1$ has a solution in $G$ \ifff\  $\mu_n(W)=1$  has a solution in~$H$.
\end{thm}

We remark that the embedding $\mu_n : G \to H$  of  Theorem~\ref{thm1}  has additional properties that are of interest even in the case when $G$ is already a 2-generated group.  For example, a solution tuple to a quadratic equation $W=1$ over $G$ such that $|W|_{\X}\leq n$ may be arbitrarily long relative to the original  alphabet $\A$ whereas the equation $\mu_n(W)=1$  has a relatively short solution tuple in $H$ with respect to the  alphabet $\{ h_1, h_2\}$. This and other technical properties of the embedding $\mu_n$, that could be useful for potential
future applications, are recorded in the following.

\begin{thm}\label{thm2}
The embedding  $\mu_n : G \to H$  of  Theorem~\ref{thm1} can be constructed in such a way that
$\mu_n$ has the following properties.

$\rm{(a)}$  Fix an enumeration $W_1 =1$, $W_2 =1, \dots$  of all quadratic equations over $G$ such that, for every $i \ge 1$, $| W_i|_\X \le n $ and $W_i =1$ has a solution in $G$.  Then there is a constant $C >0$ such that, for every $i \ge 1$, there exists a solution tuple to the equation $\mu_n(W_i) =1$ over $H$ whose length, in generators $h_1, h_2 $ of $H$,  does not exceed  $C n^4 i$.

$\rm{(b)}$  Assume that the presentation  \eqref{e0}  for $G$ is recursively enumerable. Then  defining
relations of the 2-generated  group $H = \langle h_1, h_2 \rangle$  can  be  recursively enumerated.

$\rm{(c)}$  Assume that the presentation  \eqref{e0}  for $G$ is decidable and there is an algorithm that detects whether  a quadratic equation over $G$ of length at most $n$ has a solution in $G$. Then the 2-generated  group $H = \langle h_1, h_2 \rangle$ has a decidable set of defining relations and the embedding $\mu_n : G \to H$ can be effectively constructed.
\end{thm}

As an example of a quadratic equation,  consider the  equation  $x U_1 x^{\e} U_2  = 1$, where $\e = \pm 1$ and $U_1, U_2$ are some reduced (or possibly empty if $\e = 1$) words over $\A^{\pm 1}$. Note that if $\e =-1$ then this equation has a solution \ifff\
the elements of $G$, represented by the words  $U_1, U_2^{-1}$, are conjugate in $G$. If $\e =1$, then this equation has a solution \ifff\ the element of $G$, represented by the word  $U_1^{-1} U_2$,  is a square in $G$, i.e., there is a word $T$ over $\A^{\pm 1}$ with $U_1^{-1} U_2 \overset G = T^2$.     According to  Theorem~\ref{thm1} applied with $n=2$,
if $G$ is a finite or infinite countable group, then $G$ embeds into a 2-generated group $H$, $\mu_2 : G \to H$, in which two elements of $\mu_2(G)$ are conjugate  \ifff\  they are  conjugate in $G$ and every element of $\mu_2(G)$ is a square in $H$  \ifff\  it is a  square in $G$. This is reminiscent of an embedding result of  Ol'shanskii and Sapir
\cite{OS} that states  that a finitely generated group $G$ with the solvable conjugacy problem  can  be embedded into a finitely presented group $K$ with the solvable conjugacy problem, $\sigma : G \to K$, in such a way that two elements of $\sigma(G)$ are conjugate in $K$ \ifff\ they are conjugate in $G$.
\smallskip

It would be of interest to find out whether Theorem~\ref{thm1} generalizes to arbitrary equations of bounded length
and whether one could drop the upper bound on the length of quadratic equations in Theorem~\ref{thm1}.
The first question seems to be technically relevant to the following interesting problem.

\begin{Prb} For given integer $n >0$, does there exist a real number $\lambda >0$  such that if a presentation \eqref{e0} satisfies the small cancelation condition $C'(\lambda)$, for every relation $R=1$ of  \eqref{e0}, $|R| >\lambda^{-1}$ and $R$ is not a proper power,
then every equation $W=1$ over $G$ of total length $|W| \le n$  has a solution in $G$ \ifff\ the equation $W=1$, considered as an equation over the free group $F(\A)$, has a solution in $F(\A)$?
\end{Prb}

We remark that for quadratic equations of total length $\le n$ this problem would likely have a positive solution and a proof would be analogous to the proof of Theorem~\ref{thm1} with additional consideration of contiguity subdiagrams between boundary paths of faces of type F3 and boundary paths of a surface diagram $\D$.
We also mention that the arguments of Frenkel and Klyachko \cite{FK}, which are used to prove that a nontrivial commutator cannot be a proper power in a torsion-free group $G$ that satisfies the small cancelation  condition $C'(\lambda)$ with $\lambda \ll 1$, might be useful for making some progress in nonquadratic case.
\smallskip

At the suggestion of the referee, we mention that
connections between compact surfaces and solutions of quadratic equations in free groups, free products and in hyperbolic groups were first studied by Culler \cite{Cq} and Ol'shanskii \cite{Olq}.
Earlier work on quadratic equations in free groups and in free products was done by
Edmunds \cite{E1}, \cite{E2}, Comerford and Edmunds \cite{CE}, see also articles cited in \cite{E1}, \cite{E2}, \cite{CE}. The bound of Theorem 1.2(a) is reminiscent of bounds on the length of a minimal solution of quadratic
equations in free groups obtained by Lysenok and Myasnikov \cite{LM} and by Kharlampovich and Vdovina \cite{KV}.

\section{Group Presentations and Diagrams}

Fix an even integer $n \ge 2$. Since we consider quadratic equations  $W =1$ of length $ | W|_\X \le n $, we may assume that the cardinality of $\X$ is $n$,  $| \X | = n$. Since $G$ is finite or countably infinite, we can choose an enumeration
\begin{equation}\label{enum0}
W_1 =1 , \ W_2 =1, \ldots,
\end{equation}
of all  quadratic equations over $G$ such that, for every $i \ge 1$,   $| W_i|_\X \le n $ and $W_i =1$ has a solution in $G$. Let $\cup_{i=1}^\infty  \X_i$  be an infinite countable alphabet consisting of disjoint copies $\X_i, i =1,2, \dots$, of $\X$. Let  $W_i(\X_i)$ denote the word over
the alphabet   $\A^{\pm 1} \cup \X_i^{\pm 1}$  obtained by rewriting $W_i$ so that every letter $b \in \A^{\pm 1}$  of $W_i$ is unchanged and every letter  $y$ of $W_i$, such that $y \in \X^{\pm 1}$,  is  replaced with $\beta_i(y) \in  \X_i^{\pm 1}$, where $\beta_i : \X^{\pm 1} \to  \X_i^{\pm 1}$ is a bijection such that $\beta_i( \X) = \X_i$ and   $\beta_i(x^{-1}) = \beta_i(x)^{-1}$ for every $x \in \X$.

Consider the following  group presentation
 \begin{equation}\label{e1}
    \GG_1 =  \langle \,  \cup_{i=1}^\infty \X_i  \cup  \A   \,  \|  \, R_1=1, \,  R_2=1, \dots , \,   W_1(\X_1) =1 , \,  W_2(\X_2) =1  , \dots  \,  \, \rangle
\end{equation}
whose generating set is $\cup_{i=1}^\infty \X_i  \cup  \A$ and whose  defining relations are those of \er{e0}  and  $W_i(\X_i) =1, \, i=1,2,\dots$.

 \begin{lem}\label{lem1}  There is a   natural embedding  of the   group $G$  into the group $\GG_1$  given by presentation \er{e1}, denoted $\nu_1 : G \to \GG_1$. Furthermore, if $W=1$ is an equation over  $G$ then $W=1$ has a solution in $G$ \ifff\   the  equation  $\nu_1(W)=1$  has a solution in the group $\GG_1$.
\end{lem}

\begin{proof} Denote
$\X_i =\{x_{i,1},\dots, x_{i,n} \}$ for $i=1,2, \dots$. Since the equation $W_i(\X_i) = 1$ has a solution in $G$, there exists a homomorphism
$\psi_i : G * \FF(\X_i) \to G$ such that $\psi_i$ is identical on $G$  and $\psi_i( W_i(\X_i) )=1$.  Let $U_{i,1}, \dots, U_{i,n}$ be
words over $\A^{\pm 1}$ such that $\psi_i( x_{i,j} )= U_{i,j}$ in $G$, $j=1, \dots, n$.  Then the map
$\psi_\infty( x_{i,j} ) := U_{i,j}$, where $1 \le j \le n$, $i=1,2, \dots$, and $\psi_\infty(a) := a$ for all $a \in \A$ induces a homomorphism $\psi_\infty : \GG_1 \to G$  which is identical on $G$. Hence, the group $G$ embeds in $\GG_1$. The existence of this homomorphism  $\psi_\infty : \GG_1 \to G$ also implies that, for an arbitrary equation $W=1$ over $G$, the equation $W=1$ has a solution in $G$ \ifff\  the equation $\nu_1(W)=1$ over $\GG_1$ has a solution in $\GG_1$.
\end{proof}

Denote $M := 24n$. For every  $i \ge 1$, consider a word $V_i$ over the alphabet  $\{ h_1, h_2 \}$  defined by the formula
\begin{align}\label{e2}
 V_i = V_i(h_1, h_2)  :=  h_1 h_2^{M i+1 }   h_1 h_2^{M i+2 } \dots h_1 h_2^{M (i+1)-1} h_1 h_2^{M (i+1) }  h_1   .  \end{align}

The literal (or letter-by-letter) equality of two words $U, V$  is denoted $U \equiv V$.
In the following lemma, we establish a small cancelation  condition for the words $V_i$, $i=1,2\dots$.

\begin{lem}\label{lem2}  Let $U$ be a subword of both words $V_i$ and $V_j$, defined by \er{e2}, so
$V_i \equiv V_{i,1} U V_{i,2}$ and $V_j  \equiv V_{j,1} U V_{j,2}$.
Then either $|U| < \tfrac  4 M \min \{ | V_i| , | V_j | \}$ or   $i=j$ and $V_{i,1} \equiv  V_{j,1}$.
\end{lem}
\begin{proof}   Suppose that $U$ is a subword of the word $V_i$, where $i =1,2,\dots$,   and $|U| \ge \tfrac  4 M  | V_i|$. Then
\begin{align}\label{e2a}
    |U|  \ge \tfrac  4 M  | V_i| > 4(Mi +2) >  2M(i +1) +2   .
\end{align}
Since every maximal power of $h_2$ in $V_i$ is no longer than $M(i+1)$, it follows from \er{e2a} that $U$ contains a subword of the form  $h_1 h_2^k h_1$, where  $Mi+1 \le k \le M(i+1)$.  Now our claim follows from the fact
that each word $V_1, V_2, \dots$ contains a unique subword of the form $h_1 h_2^k h_1$, where $Mi+1 \le k \le M(i+1)$.
\end{proof}

Let $\cup_{i=1}^\infty \X_i  = \{ x_1, x_2, \dots \}$ be an enumeration of elements of $\cup_{i=1}^\infty \X_i$ compatible with the enumeration of sets $\X_i$, i.e., if $x_j \in \X_k$, $x_{j'} \in \X_{k'}$ and $k <k'$, then  $j <j'$.
Using this enumeration,  new generators  $h_1, h_2 $ and the words $V_i(h_1, h_2)$, we extend the presentation \er{e1}  as follows
 \begin{align}\notag
    \GG_2 =  & \langle \,  \cup_{i=1}^\infty \X_i  \cup  \A \cup  \{ h_1, h_2 \}   \,  \|  \, R_1=1, \,  R_2=1, \dots , \,  W_1(\X_1) =1 ,
      \\   \label{e3} &       W_2(\X_2) =1  , \,  \dots  ,  \,    x_i V_{2i}^{-1} =1 ,   \,  a_i  V_{2i+1}^{-1} = 1 , \,  i=1,2, \dots  \, \rangle .
\end{align}

To study this group presentation and quadratic equations over $\GG_2 $, we will use diagrams over the presentation \er{e3}. We start with basic definitions.

Let $\D$ be a finite 2-complex and let $\D(i)$ denote the set of closures of   $i$-cells  of $\D$, $i=0,1,2$. The elements of  $\D(i)$ are called {\em vertices, edges, faces} of $\D$ if  $i=0, 1, 2$, resp.
We also consider the set $\vec \D(1)$ of oriented  1-cells of  $\D$.
If $e \in  \vec \D(1)$,  then $e^{-1} $ denotes $e$ with opposite orientation.
For every $e \in \vec \D(1)$,  let $e_-$, $e_+$ denote the initial, terminal, resp., vertices of $e$.
In particular, $(e^{-1})_- = e_+$ and $(e^{-1})_+ = e_-$.  Note that $e \ne e^{-1}$.

A path $p = e_1 \dots e_\ell$ in $\D$ is a sequence of oriented edges $e_1, \dots, e_\ell$ of $\D$ with
$(e_i)_+ = (e_{i+1})_-$, $i =1,\dots, \ell-1$. The length of a path $p= e_1 \dots e_\ell$ is $|p| = \ell$.
The initial vertex of $p$ is $p_- := (e_1)_-$ and the terminal vertex of $p$ is $p_+ := (e_\ell)_+$. A path $p$ is called {\em closed} if $p_- = p_+$. A path $p$ is called {\em reduced} if $|p|>0$ and  $p$ contains no subpath of the form $e e^{-1}$, where $e$ is an  edge.  A {\em cyclic} path is a closed path with no distinguished initial vertex.  A path $p= e_1 \dots e_\ell$ is called {\em simple} if the  vertices $(e_1)_-, \dots,   (e_\ell)_-, (e_\ell)_+$ are all distinct. A closed path is  {\em simple} if the  vertices $(e_1)_-, \dots,   (e_\ell)_-$  are all distinct.

A {\em  diagram} $\D$ over presentation \eqref{e3} is a connected finite 2-complex which is equipped with a labeling function
$$
\ph : \vec \D(1) \to  \cup_{i=1}^\infty \X_i^{\pm 1}  \cup  \A^{\pm 1} \cup  \{ h_1^{\pm 1}, h_2^{\pm 1}, 1 \}
$$
such that, for every $e \in \vec \D(1)$, one has $\ph(e^{-1}) = \ph(e)^{-1}$, where $1^{-1} := 1$, and, for every face $\Pi$ of $\D$, if $\p \Pi = e_1 \dots e_\ell$ is a boundary path of $\Pi$, where $e_1, \dots, e_\ell \in \vec \D(1)$,
then the label  $\ph(\p \Pi) :=   \ph(e_1) \dots \ph(e_\ell)$  of  $\p \Pi$ has   one of the following three forms.
\begin{enumerate}
\item[(F1)] $\ph(\p \Pi) = 1^\ell$.
\item[(F2)]  $\ell= 4$ and $\ph(\p \Pi)$ is a cyclic permutation of a word $y 1 y^{-1} 1$, where $y \in  \cup_{i=1}^\infty \X_i  \cup  \A \cup  \{ h_1, h_2 \}$.
\item[(F3)] $\ph(\p \Pi)$ is a cyclic permutation of one of the words $R^{\pm 1}$, where $R = 1$ is a relation of the presentation \er{e3}.
\end{enumerate}

A face $\Pi$ of $\D$ is said to have {\em type F1, F2, F3}  if $\ph(\p \Pi)$ has the form (F1), (F2), (F3), resp.
The set of faces of type $Fj$ is denoted $\D_j(2)$, $j =1,2,3$.

An edge $e \in \vec \D(1)$ is called an {\em  $a$-edge, $x$-edge, $h$-edge, $1$-edge } if
$\ph(e) \in \A^{\pm 1}$, $\ph(e) \in \cup_{i=1}^\infty  \X_i^{\pm 1}$, $\ph(e) \in \{ h_1^{\pm 1},  h_2^{\pm 1}  \}$,
$\ph(e) = 1$, resp.   An edge $e \in \vec \D(1)$ is termed  {\em  essential } if $e$ is not a $1$-edge.

We will say that $\D$ is a {\em surface diagram of type $(k, k')$} over \eqref{e3}   if $\D$ is a diagram over \eqref{e3} and    $\D$,  as a topological space, is homeomorphic to a compact (orientable or nonorientable) surface  that has Euler characteristic $k$ and contains $k'$ punctures.
This surface is called the {\em underlying} surface for $\D$.
In particular, $\D$ is called  a {\em disk} diagram if $\D$ is a surface diagram  of type $(1, 1)$, hence, the underlying surface for $\D$ is a disk.

If $\D$ is a surface diagram and the underlying surface is orientable, then a fixed orientation of  the underlying surface makes it possible to define
positive (=counterclockwise) and  negative (=clockwise) orientation for  boundaries of faces of $\D$ and for connected components of $\p \D$.
Regardless of whether the underlying surface is orientable or not, we always consider the boundary $\p \Pi$ of  a face $\Pi$ of $\D$  or a  connected component $c$ of the boundary $\p \D$ of $\D$  as a cyclic  path which is called  a {\em boundary path} of $\Pi$ or  a {\em boundary path} of $\D$, resp.   Note that  $(\p \Pi)^{-1}$ or $c^{-1}$   are also  boundary paths of $\Pi$ or $\D$, resp.,  with the  opposite orientation.

Suppose that $\D$ is a surface diagram over  \er{e3}. Making refinements of $\D$ by using faces of type F1, F2
if necessary (informally, we ``thicken" boundary paths of faces of type F3 and $\p \D$, this should be evident; more formal details can be found in \cite{Ol}), we may assume that the following property holds for $\D$.

\begin{enumerate}
\item[(A)]  Suppose that each of $c_1, c_2$ is either  a boundary path of a face of type F3 in $\D$ or a  boundary path  of $\D$. Then $c_1, c_2$ are closed simple paths  and  either $c_1$ is a cyclic permutation of one of $c_2$, $c_2^{-1}$ or $c_1, c_2$ have no common vertices.
\end{enumerate}

Note that the property (A) implies that if an essential edge $e$ of $\D$ belongs to  a boundary path of  a face of type F3 or  $e$ belongs to  a boundary path of  $\D$, then $e$ also belongs to  a boundary path of a face of type F2.

From now on we always assume, unless stated otherwise, that a diagram is a surface diagram over \er{e3} with
the property (A).

Recall that the literal (or letter-by-letter) equality of the words $U, V$  is denoted $U \equiv V$.

\begin{lem}\label{vk}  Let  $W$ be a nonempty word  over  the alphabet
$$\cup_{i=1}^\infty \X_i^{\pm 1}  \cup  \A^{\pm 1} \cup  \{ h_1^{\pm 1}, h_2^{\pm 1}, 1 \} $$ and let $\GG_2$ be the group defined by presentation \er{e3}. Then $W \overset {\GG_2} = 1$ if and only if there is a surface diagram  $\D$ of type $(1,1)$, called a {\em disk diagram},   over  presentation \er{e3} such that $\ph( \p \D ) \equiv W$.
\end{lem}

\begin{proof} The proof is straightforward, for details the reader is referred to \cite{Ol}, \cite{Iv94}, see also  \cite{LS}. As in \cite{Ol},  faces of type  F1, F2 make it possible to ``thicken" the diagram and turn its underlying topological space into a disk.
\end{proof}

Suppose that $\Psi$ is a finite graph on a compact surface $S$.
Consider the following property of $\Psi$ in which $m \ge 2$ is an integer parameter.

\begin{enumerate}
\item[(B)]  If  $f$ is an oriented edge of $\Psi$ with $f_-=f_+$ then the  edge $f$  does not bound a disk on $S$ whose interior contains no vertices of $\Psi$. Furthermore, if $f_1, \dots, f_m$ are oriented edges of $\Psi$ such that $(f_i)_- =(f_j)_-$ and $(f_i)_+ =(f_j)_+$ for all $i,j=1,\dots, m$, then it is not true that each path $f_1f_2^{-1}$, $f_2 f_3^{- 1}, \dots$, $f_{m-1}f_m^{-1}$ bounds a disk on $S$ whose interior contains no vertices of $\Psi$.
\end{enumerate}

We finish this section with a lemma about graphs on surfaces.

\begin{lem}\label{gr1}
Let  $S$ be a compact surface  whose Euler characteristic is $\chi(S) =k$ and let $\Psi$ be a finite graph on $S$ that has the property (B) with parameter $m=2$. If $V_{\Psi}$ and $E_{\Psi}$ denote the number of vertices and nonoriented edges of $\Psi$, resp., then $E_{\Psi} \le 3( V_{\Psi} -k)$.
\end{lem}

\begin{proof} Note that the property (B) with parameter $m=2$ can be stated less formally by saying
that the partial cell decomposition of $S$, defined by the graph $\Psi$, contains no 1- and 2-gons whose interiors contain  no vertices of $\Psi$.
Preserving this condition, i.e., preserving the property  (B) with parameter $m=2$, we will draw as many new edges in $\Psi$ as possible and obtain a  graph  $\Psi'$ with  $V_{\Psi'} = V_{\Psi}$, $E_{\Psi'} \ge  E_{\Psi}$. Note that $\Psi'$ is connected and if $c$ is a \cncm\ of $\p \D$ then there is a closed simple path $e_{c,1} \dots e_{c,k_c}$, where $e_{c,1}, \dots, e_{c,k_c}$ are edges of $\Psi'$, such that  $e_{c,1} \dots e_{c,k_c}$ and $c$ bound an annulus $A_c$ whose interior contains no vertices of $\Psi'$. Hence,  taking  $A_c$ out of $S$ and adding back the cycle $e_{c,1} \dots e_{c,k_c}$ for every  \cncm\  $c$ of $\p \D$, we obtain a surface $S'$ such that $\chi(S') =\chi(S) =k$. In addition, it follows from definitions that $S' \setminus \Psi'$ is a collection of open disks.
Indeed, if a \cncm\  of  $S' \setminus \Psi'$ were different from a disk, then one could draw an additional edge
in $\Psi'$ without creating a 1- or 2-gon, contrary to the maximality of $\Psi'$. Hence, the graph  $\Psi'$ defines a  cell decomposition of $S'$ and
\begin{equation}\label{Ech}
V_{\Psi'} - E_{\Psi'}  + F_{\Psi'} =   \chi(S') = k ,
\end{equation}
where $F_{\Psi'}$ is the number of faces of the cell decomposition of  $S'$ defined by $\Psi'$. Since there are no
1- and 2-gons in this decomposition, every face has 3 edges in its boundary path  which implies that $3 F_{\Psi'} \le 2E_{\Psi'} $ or $F_{\Psi'} \le \tfrac 2 3 E_{\Psi'}$. Hence, it follows from \er{Ech} that $V_{\Psi'} -\tfrac 1 3 E_{\Psi'} \ge k$ or $E_{\Psi'}  \le 3(V_{\Psi'} - k)$. Since $V_{\Psi'} = V_{\Psi}$, $E_{\Psi'} \ge  E_{\Psi}$, our claim is proved.
\end{proof}

\section{Contiguity  Subdiagrams}

As in Sect. 2, let $\D$ be a surface diagram over presentation \er{e3} with
 property (A).
Consider a relation $\sim_2$ on the set $\D_2(2)$  of faces
of type F2 so that $\Pi_1 \sim_2 \Pi_2$ \ifff\ there is an essential edge $e$  such that $e$ belongs to $(\p \Pi_1)^{\pm 1} := \p \Pi_1\cup \p \Pi_1^{- 1}$ and  $e$ belongs to $(\p \Pi_2)^{\pm 1}$. It is easy to see that this relation is reflexive and symmetric on $\D_2(2)$.
The transitive closure of this relation $\sim_2$ is an equivalence relation on $\D_2(2)$ which we denote by $\sim$. Let $[\Pi]_{\sim}$  denote the equivalence class of  a face $\Pi$ of type F2  relative to this  equivalence relation. For every $\Pi \in \D_2(2)$, we consider a minimal
subcomplex   $\BT_\Pi = \BT ( [\Pi]_{\sim})$ of  $\D$ that contains all faces of $[\Pi]_{\sim}$.
It follows from definitions that there exists a surface diagram $\AL_\Pi$ of type (1,1) (meaning that  $\AL_\Pi$ is a disk) or of type (0,1) (meaning that $\AL_\Pi$ is an annulus) and a continuous cellular map
$
\mu_\Pi :  \AL_\Pi \to \BT_\Pi
$
such that $\mu_\Pi$ preserves dimension of cells, $\ph$-labels of edges, and $\mu_\Pi(\AL_\Pi) = \BT_\Pi$.   We also require that
$\AL_\Pi$ consists of faces of type F2 and their number $| \AL_\Pi(2)|$ equals the number $| \BT_\Pi(2)|$ of faces in
$\BT_\Pi$. Note that $\mu_\Pi$ need not be injective and this is the reason we consider an ``ideal" preimage $\AL_\Pi$ of  the subcomplex $\BT_\Pi$.

If $\AL_\Pi$ is a disk, then  $\p \AL_\Pi = s_1 f_1 s_2 f_2$, where $f_1, f_2$ are essential edges with $\ph(f_1) = \ph(f_2)^{-1} \ne 1$,  and $s_1, s_2$  are simple  paths consisting of 1-edges with $|s_1| = |s_2| = | \AL_\Pi(2)|$, see Fig.~1(a). In this case, we say that $\BT_\Pi$ is a {\em band} between the edges $e_1$, $e_2$ and that $\p \BT_\Pi =  u_1 e_1u_2 e_2 $, where $e_i = \mu_\Pi (f_i )$, $u_i = \mu_\Pi (s_i )$, $i =1,2$, is a {\em standard boundary path} of the band $\BT_\Pi$.  Clearly,  $e_1, e_2$ are essential edges with  $\ph(e_1) = \ph(f_1) =\ph(e_2)^{-1}  \ne 1$ and $|u_1 | = |s_1| = | u_2 |$ but $u_1, u_2$ need not be simple paths. If $\ph(e_1)^{\pm 1} = y$, where $y \in  \cup_{i=1}^\infty \X_i  \cup  \A \cup  \{ h_1, h_2 \}$, then we may also specify that   $\BT_\Pi$ is a $y$-band.

Since we  neither fix a base vertex for $\p \BT_\Pi$, nor fix an orientation for $ \BT_\Pi$, it follows that if
$\p \BT_\Pi =  u_1 e_1u_2 e_2 $ is a  standard boundary path for  a band $\BT_\Pi$, then $u_2 e_2  u_1 e_1$ and $u_2^{-1} e_1^{-1}  u_1^{-1} e_2^{-1}$ are also standard boundary paths for $\BT_\Pi$.  We also observe that a standard  boundary path of a band $\BT$ need not be the topological boundary of $\BT$ but it can be turned into the topological boundary (of a deformed space) by an arbitrarily small deformation of $\BT$ which pushes $\BT$ into its interior.

On the other hand, if $\AL_\Pi$ is an annulus, then $\p \AL_\Pi =  s_1  \cup  s_2$, where $s_1, s_2$ are cyclic simple  paths consisting of 1-edges, $|s_1| = |s_2| = | \AL_\Pi(2)|$, see Fig.~1(b). In this case, we say that $\BT_\Pi$ is  an annulus and that $\p \BT_\Pi =  u_1 \cup   u_2$, where  $u_i = \mu_\Pi (s_i )$, $i =1,2$, are  boundary paths of the  annulus $\BT_\Pi$.

\begin{center}
\begin{tikzpicture}[scale=.324]
\draw  (-2,-5) rectangle (0.5,1);
\draw (-2,0) -- (0.5,0);
\draw (-2,-1) -- (0.5,-1);
\draw (-2,-2) -- (0.5,-2);
\draw (-2,-3) -- (0.5,-3);
\draw (-2,-4) -- (0.5,-4);
\node at (-3,-2) {$s_1$};
\node at (1.5,-2) {$s_2$};
\node at (-.75,1.7) {$f_1$};
\node at (-0.75,-5.7) {$f_2$};
\node at (-.75,-7) {Fig. 1(a)};
\node at (10,-7) {Fig. 1(b)};
\draw  (10,-2)  circle (1);
\draw  (10,-2) circle (2.5);
\draw (10,-3) -- (10,-4.5);
\draw (10,-1) -- (10,0.5);
\draw (11,-2) -- (12.5,-2);
\draw (9,-2) -- (7.5,-2);
\node at (10,-2.35) {$s_1$};
\node at (7.5,-4) {$s_2$};
\node at (-3.5,-4.5) {$\AL_\Pi$};
\node at (7.5,0.5) {$\AL_\Pi$};
\draw  (-2,1) [fill = black] circle (.06);
\draw  (.5,1) [fill = black] circle (.06);
\draw  (-2,-5) [fill = black] circle (.06);
\draw  (.5,-5) [fill = black] circle (.06);
\end{tikzpicture}
\end{center}

Note that if $\BT$  is a band  and  $\p \BT=u_1e_1  u_2 e_2 $ is a standard boundary path of  $\BT$, then each of the essential edges $e_1$, $e_2$   belongs  either to a boundary path of $\D$  or to a boundary path  of a face of type F3.  If, say, $e_i$ belongs to $c_i$,  where $i=1,2$ and $c_i$ is a boundary path of $\D$ or is a boundary path  of a face of type F3, then we say that $\BT ( [\Pi]_{\sim})$ is a band between $c_1$ and $c_2$.

Let $\BT$ be a band between edges $e_1$ and $e_2$. Let $o_1 \in e_1$, $o_2 \in e_2$ be  interior points of
edges $e_1, e_2$ and let $\ell(\BT)$  be a simple arc such that  $\ell(\BT)$  is contained in $\BT$, the boundary points of
$\ell(\BT)$ are $o_1$,  $o_2$ and the intersection of $\ell(\BT)$ with every face $\Pi$ of $\BT$ consists of a single
arc which is properly embedded in $\Pi$ and the boundary points of the arc are interior points of essential edges of $\p \Pi$. Such an arc  $\ell(\BT)$   is called a {\em connecting line} for $\BT$. It follows from definitions that if $\BT$ is a band between edges $e_1$ and $e_2$, then a connecting line $\ell(\BT)$ for $\BT$ connects interior points of $e_1$, $e_2$ through faces of $\BT$ of type F2.

Let $s$ be  either a subpath of $\p \Pi$ (where $\Pi$ is a face of type F3 in $\D$) or a subpath of $\p \D$ such that $s$ consists of $h$-edges and $s$ is maximal with respect to this property. Such $s$ is called an {\em $h$-section} of $\D$.

Suppose that $s_1, s_2$ are $h$-sections of $\D$, not necessarily distinct, and $\BT_1, \BT_2$ are bands between $s_1, s_2$, perhaps $\BT_1 = \BT_2$, whose standard boundary paths  are $\p \BT_i =  u_{i1} e_{i1}  u_{i2} e_{i2}$, $i=1,2$, where $e_{i1}, e_{i2}$ are essential edges of $\p  \BT_i $. Also, assume that $e_{11}, e_{21}$ are edges of $s_1$ so that $s_1 = s_{11} e_{11} s_{12} e_{21} s_{13}$ and  $e_{22}, e_{12}$ are edges of $s_2$ so that $s_2 = s_{21} e_{22} s_{22} e_{12} s_{23}$, see Fig.~2.

\vskip 3mm
\begin{center}
\begin{tikzpicture}[scale=.47]
\draw  (-5,-1.5)  rectangle (9,-5.5);
\draw (2,-1.5) -- (11.5,-1.5);
\draw (9,-5.5) -- (11.5,-5.5);
\draw (-3,-1.5) -- (-7.5,-1.5);
\draw (-1,-5.5) -- (-7.5,-5.5);
\draw (-1,-1.5) -- (-1,-5.5);
\draw (5,-1.5) -- (5,-5.5);

\draw [-latex](-6.4,-5.5) --(-6.6,-5.5);
\draw [-latex](-2.9,-5.5) --(-3.1,-5.5);
\draw [-latex](2.1,-5.5) --(1.9,-5.5);
\draw [-latex](7.1,-5.5) --(6.9,-5.5);
\draw [-latex](10.6,-5.5) --(10.4,-5.5);

\draw [-latex](-6.6,-1.5) --(-6.4,-1.5);
\draw [-latex](-3.1,-1.5) --(-2.9,-1.5);
\draw [-latex](1.9,-1.5) --(2.1,-1.5);
\draw [-latex](6.9,-1.5) --(7.1,-1.5);
\draw [-latex](10.4,-1.5) --(10.6,-1.5);

\draw [-latex](-5,-3.6) --(-5,-3.4);
\draw [-latex](5,-3.6) --(5,-3.4);
\draw [-latex](-1,-3.4) --(-1,-3.6);
\draw [-latex](9,-3.4) --(9,-3.6);

\draw  (-5,-1.5) [fill = black] circle (.06);
\draw  (-1,-1.5) [fill = black] circle (.06);
\draw  (5,-1.5) [fill = black] circle (.06);
\draw  (9,-1.5) [fill = black] circle (.06);
\draw  (-5,-5.5) [fill = black] circle (.06);
\draw  (-1,-5.5) [fill = black] circle (.06);
\draw  (5,-5.5) [fill = black] circle (.06);
\draw  (9,-5.5) [fill = black] circle (.06);

\node at (-5.7,-3.5) {$u_{11}$};
\node at (-0.3,-3.5) {$u_{12}$};
\node at (-3,-.9) {$e_{11}$};
\node at (-3,-6.1) {$e_{12}$};
\node at (2,-.9) {$s_{12}$};
\node at (2,-6.1) {$s_{22}$};
\node at (-6.5,-.9) {$s_{11}$};
\node at (-6.5,-6.1) {$s_{23}$};
\node at (10.5,-.9) {$s_{13}$};
\node at (10.5,-6.1) {$s_{21}$};
\node at (4.3,-3.5) {$u_{21}$};
\node at (9.7,-3.5) {$u_{22}$};
\node at (7,-.9) { $e_{21}$};
\node at (7,-6.1) {$e_{22}$};
\node at (-3,-3.5) {$\BT_1$};
\node at (7,-3.5) {$\BT_2$};
\node at (2,-7.7) {Fig. 2};
\end{tikzpicture}
\end{center}

\noindent
Note that the path
$
p =u_{11} e_{11} s_{12} e_{21} u_{22}e_{22} s_{22} e_{12}
$
is closed.  Furthermore, assume  that there exists a connected subcomplex $\Gamma'$ of $\D$ such that $\Gamma'$ contains  $\BT_1, \BT_2, p$,  $\Gamma'$  has no faces of type F3 with $h$-edges, and the path $p$ is nullhomotopic in $\Gamma'$. Then we  consider  a minimal (relative to the inclusion relation) such  subcomplex $\Gamma$ whose boundary path $\p \Gamma$ (up to arbitrarily small deformation; this time we skip introduction of an ``ideal" disk diagram whose image is $\Gamma$) can be written in the form  $\p \Gamma =  u_{11} (e_{11} s_{12} e_{21}) u_{22}   (e_{22} s_{22} e_{12} )$. Note that if $\BT_1 = \BT_2$, then $\Gamma :=  \BT_1 $ and   $\p \Gamma = \p \BT_1 = u_{11} e_{11}  u_{12}  e_{12}$.
Such a  subcomplex $\Gamma$ of $\D$  is unique and is  called a {\em \cntsd\ }  between $h$-sections    $s_1$ and $s_2$ defined by the bands $\BT_1$, $\BT_2$.  Denote $\Gamma \wedge s_1 :=  e_{11} s_{12} e_{21}$ and $\Gamma \wedge s_2 := e_{22} s_{22} e_{12} $ and call these paths  {\em contiguity arcs} of  $\Gamma$.  If  $\BT_1 = \BT_2$, then $\Gamma \wedge s_1 :=  e_{11}$ and $\Gamma \wedge s_2 := e_{12} $.  Since $\Gamma$ contains no faces of type F3 with $h$-edges, $s_1, s_2$ are $h$-sections and $u_{11}, u_{12}$ consist of 1-edges, it follows that  $\ph( e_{11} s_{12} e_{21}) \equiv \ph(e_{22} s_{22} e_{12})^{-1}$ and, by definitions and property (A), there exists a simple path $t$, $|t| > 0$,  that connects $(u_{11})_- \in s_2$ with $(u_{11})_+ \in s_1$ and consists of 1-edges.
A factorization of $\p \Gamma$ of the form
$$
\p \Gamma =  u_{11} (e_{11} s_{12} e_{21}) u_{22}   (e_{22} s_{22} e_{12} )
$$
is called a  {\em standard boundary path} of the   \cntsd\   $\Gamma$.

A \cntsd\    $\Gamma$ between $h$-sections  $s_1$, $s_2$ is called  {\em maximal} if there is no
\cntsd\    $\Gamma'$ between  $s_1$, $s_2$ such that $\Gamma \wedge s_i$ is a subpath of $\Gamma' \wedge s_i$, for both  $i=1,2$,  and $| \Gamma \wedge s_1 | + | \Gamma \wedge s_2 | < | \Gamma' \wedge s_1 | + | \Gamma' \wedge s_2 |$.

In the following lemma, we record simple facts about bands and \cntsd s.

\begin{lem}\label{cntsd}  Suppose that $e$ is an edge of an $h$-section of a surface diagram $\D$  and $\BT$ is  an $h$-band in $\D$. Then the following are true.

$(a)$  There is an  $h$-band one of whose essential edges is $e$.

$(b)$  There is a unique maximal \cntsd\  $\Gamma$ that contains $\BT$.

$(c)$  There is a unique maximal \cntsd\  one of whose contiguity arcs contains $e$.
\end{lem}

\begin{proof} (a)  Suppose that $e$ belongs to  a boundary path of  $\Pi$, where $\Pi$ is a face of type F3 in $\D$.  Then it follows from property (A) that if $o$ is an interior point of $e$ then a \rgnb\ $N$ of $o$ in $\D$ consists of two parts separated by the arc $N \cap e$, one of which is in $\Pi$ and the other  of which is in a face $\Pi'$ of type F2. Then
$\BT_{\Pi'}$ is a desired $h$-band. If $e$ is on  $\p \D$ then, again by property (A), there is a face  $\Pi''$ of type F2 whose  boundary path contains $e$. Then $\BT_{\Pi''}$ is a desired $h$-band.

(b) Let $\BT$ be a band  between $h$-sections $s_1, s_2$. Then there exists a \cntsd\ $\Gamma$ between $s_1$ and $ s_2$ that contains $\BT$. For example, $\Gamma = \BT$.  If $\Gamma_1$, $\Gamma_2$ are two \cntsd s between $s_1$ and $ s_2$ that contain $\BT$, then it is easy to check that there is also a \cntsd\ $\Gamma_0$ that contains both $\Gamma_1$ and $\Gamma_2$.
This implies the  uniqueness of a maximal \cntsd\  that contains $\BT$.

(c) This follows from parts (a)--(b).
\end{proof}

Let $\D$ be a surface diagram over presentation \er{e3} of type $(k, k')$. Consider the set $\C_h$  of all maximal
\cntsd s between $h$-sections in $\D$. It follows from Lemma~\ref{cntsd} that, for every edge $e$ of an $h$-section $s$ of
$\D$,  there is a unique maximal \cntsd\  $\Gamma \in \C_h$  whose contiguity arc contains $e$, i.e., $e$ belongs to
$\Gamma \wedge s$.

For every $\Gamma \in \C_h$, we pick a connecting line $\ell(\BT)$, where $\BT = \BT(\Gamma)$ is a band that defines $\Gamma$. Denote  $\ell(\Gamma) := \ell(\BT)$ and call  $\ell(\Gamma)$ a {\em connecting line} of $\Gamma$.  For every face $\Pi$ of type F3,  whose boundary path $\p \Pi$ contains $h$-edges, we pick a vertex $v_\Pi$ in the interior of $\Pi$. Then we connect each point in $(\cup_{\Gamma \in \C_h }  \ell(\Gamma) ) \cap \p \Pi$  to $v_\Pi$ by drawing simple arcs in $\Pi$ such that the arcs' pairwise intersections are $\{ v_\Pi \}$ and each arc intersects $\p \Pi$ only at its endpoint different from $v_\Pi$. The union of all such arcs and connecting lines $\ell(\Gamma)$, $\Gamma \in \C_h$, is a graph on $\D$, denoted $\Psi_h$, whose vertex set is the union of the set $\{ v_\Pi  \mid  \Pi \in \D_3(2), \ \p \Pi \ \mbox{has} \  h\mbox{-edges} \}$ and the set of those boundary points of connecting lines $\ell(\Gamma)$, $\Gamma \in \C_h$, that belong to $\p \D$. Note that the set of nonoriented edges of $\Psi_h$ is in bijective correspondence with the set $\C_h$ of maximal contiguity subdiagrams and that each edge  of $\Psi_h$  is obtained from $\ell(\Gamma)$, where $\Gamma \in \C_h$,    by extending $\ell(\Gamma)$ into a face $\Pi$ of type F3 whenever a point of $\p \ell(\Gamma)$ belongs to  $\p \Pi$.
\medskip

Now we will define reduced diagrams over the presentation \er{e3}.
We say that a pair of distinct faces $\Pi_1, \Pi_2$ of type F3 with $h$-edges in a surface diagram $\D$ over \er{e3} forms a {\em reducible pair} if there is a simple path $t$ such that  $t$ connects some vertices $t_- \in \p \Pi_1$, $t_+ \in \p \Pi_2$,
$t$ consists of 1-edges, $|t| >0$,  and the label  $\ph(\p \Gamma )$ of the boundary path  $\p \Gamma  = t \p \Pi_2 t^{-1} \p \Pi_1$ of the    subdiagram $\Gamma$, consisting of $t, \Pi_1, \Pi_2$,  is equal to 1 in the free group whose free base is the alphabet $\cup_{i=1}^\infty \X_i \cup \A \cup \{ h_1, h_2 \}$, see Fig.~3.

\vskip 2mm
\begin{center}
\begin{tikzpicture}[scale=.59]
\draw  (-4.5,4) ellipse (1 and 1);
\draw  (1,4) ellipse (1 and 1);
\draw  plot[smooth, tension=.6] coordinates {(-3.5,4) (-2,5) (-1,3.5) (0,4)};
\draw  (-3.5,4) [fill = black] circle (.05);
\draw  (0,4) [fill = black] circle (.05);
\node at (-4.5,4) {$\Pi_1$};
\node at (1,4) {$\Pi_2$};
\node at (-2.24,4.5) {$t$};
\draw [-latex](-2.12,4.991) --(-2.08,5);
\node at (-1.8,2.2) {Fig. 3};
\end{tikzpicture}
\end{center}
\vskip 2mm

It is easy to see that if  $\Pi_1, \Pi_2$ form a reducible pair in $\D$, then one can perform a surgery on $\D$ that replaces  the subdiagram $\Gamma$, whose boundary path is $\p \Gamma  = t \p \Pi_2 t^{-1} \p \Pi_1$, by a subdiagram that consists of faces of type F1--F2.  If $\D'$ is obtained from $\D$ by this surgery, then $\ph(\p \D')$ is identical to $\ph(\p \D)$ (in fact, the surgery does not affect  the boundary of $\D$) and $| \D'_3(2) | = | \D_3(2) |-2$. Hence, by induction on the number $| \D_3(2) |$ of faces of type F3, every  diagram $\D$ can be turned into a diagram $\bar \D$ without reducible pairs and with no change in $\ph(\p D)$.
A  diagram $\D$ will be called {\em reduced} if $\D$  contains no reducible pairs.

\begin{lem}\label{gr2}  Suppose that $\D$ is a reduced surface diagram of type $(k, k')$, there are no
$h$-edges  contained in $\p \D$, $\D$ contains a face of type F3 whose boundary path has $h$-edges, and the graph $\Psi_h$ is defined as above. Then there exists a vertex in $\Psi_h$ whose degree   is  positive and is at most $\max \{ 12(1 - k ), 12 \}$.
\end{lem}

\begin{proof} Let  $v_\Pi$ be a vertex of $\Psi_h$, let $f$ be an oriented edge of  $\Psi_h$  such that $f_- = f_+ = v_\Pi$ and $f$ bounds a disk on $\D$. It follows from the definition of relations in \er{e3} that if $e_1, e_2$ are $h$-edges of $\p \Pi$, then either $\ph(e_1), \ph(e_2) \in \{ h_1, h_2 \}$  or  $\ph(e_1), \ph(e_2) \in \{ h_1^{-1}, h_2^{-1}   \}$.
On the other hand, let  $\Gamma \in \C_h$ be the  contiguity subdiagram that $f$ passes through and let
$\BT$ denote the bond that contains the connecting line $\ell(\Gamma)$. If $e_3, e_4$ are $h$-edges of $\p \BT$, then it follows from the fact that $f$ bounds a disk on $\D$ that  $\ph(e_3)= \ph(e_2)^{-1}$, hence, the inclusions $e_3, e_4 \in \p \Pi$ are impossible. Thus, there is no 1-gon in the partial cell decomposition of $\D$ defined by $\Psi_h$.

Now assume that the property (B) fails for   $\Psi_h$  with parameter $m =3$. This means that
there are three distinct edges $f_1, f_2, f_3$  in  $\Psi_h$   such that
$$
(f_1)_- = (f_2)_- =(f_3)_- = v_{\Pi} , \quad  (f_1)_+ = (f_2)_+ =(f_3)_+ = v_{\Pi'} ,
$$
where $\Pi, \Pi'$ are some faces of type F3 with $h$-edges,  such that  both  paths $f_1 f_2^{-1}$,  $f_2 f_3^{-1}$ bound  disks on $\D$ whose interiors contain no vertices of $\Psi_h$. Let $f_i$ be the extension of the connecting line $\ell(\Gamma_i)$, where $\Gamma_i \in \C_h$, $i=1,2,3$, and $s, s'$ be $h$-sections of the faces $\Pi, \Pi'$, resp. Then it is not difficult to check that either $\Gamma_1, \Gamma_2$ or  $\Gamma_2, \Gamma_3$  are contained in a \cntsd\ $\Gamma$ between $s$ and $s'$, contrary to the maximality of  \cntsd s  $\Gamma_1, \Gamma_2, \Gamma_3$. This
contradiction proves that the property (B) holds for $\Psi_h$ with $m=3$.

Consider those pairs $\{ f, f' \}$ of oriented edges of  $\Psi_h$  for which the property (B) with $m=2$ fails.
Note that the property (B) with $m=3$ for  $\Psi_h$  implies that every  oriented edge $e$ of $\Psi_h$ is
contained in at most one such pair $\{ f, f' \}$. For each such  pair $\{ f, f' \}$, we remove
edges $(f')^{\pm 1}$  (or  $f^{\pm 1}$) from  $\Psi_h$.  Doing this results in a graph  $\wht \Psi_h$ which, as follows from definitions, has the property (B) with $m=2$. Therefore, Lemma~\ref{gr1} applies to $\wht \Psi_h$ and yields that
$ E_{\wht \Psi_h} \le 3(V_{\wht \Psi_h} - k)$, where $V_{\wht \Psi_h}, E_{\wht \Psi_h}$ denote the number of vertices,  nonoriented edges, resp., in  $\wht \Psi_h$. Note that $V_{\Psi_h} = V_{\wht \Psi_h}$ and $ E_{\Psi_h} \le 2 E_{\wht \Psi_h}$. Hence, $ E_{\Psi_h} \le 6(V_{\Psi_h} - k)$.  If $d$ is the minimal positive degree of a vertex in $V_{\Psi_h}$, then it is easy to see from definitions that $d >0$ and $d V_{\Psi_h} \le 2 E_{\Psi_h}$.
Thus $d V_{\Psi_h} \le 12(V_{\Psi_h} - k)$ and
$$
d \le 12(1 - \tfrac {k}{V_{\Psi_h} } ) \le \max \{ 12(1 - k ), 12 \} ,
$$
as desired.
\end{proof}

\section{Proofs of Theorems}

{\em Proof of Theorem \ref{thm1}.}
 First we observe that the  group $\GG_2$, given by  presentation  \er{e3}, can also be presented by generators and relations in the following form
 \begin{align}\label{prh}
   \langle \,   h_1, h_2    \,  \|  \, \wht R_1=1, \, \wht R_2=1, \dots , \, \wht  W_1 =1 ,
     \wht    W_2 =1  , \,  \dots  \, \rangle ,
\end{align}
 where, for every possible $i =1,2, \dots$, the defining words $\wht R_i$, $\wht W_i$ result from rewriting  of the words $R_i$, $W_i(\X_i)$, resp., of  presentation  \er{e3} so that letters $a_{j_1}^{ \e_1}$,  $x_{j_2}^{ \e_2}$, where  $a_{j_1} \in \A$, $x_{j_2} \in \cup_{i'=1}^\infty \X_{i'}$, $\e_1, \e_2 = \pm 1$, are replaced with the words
 $V_{2j_1+1}^{ \e_1}$,  $V_{2j_2}^{ \e_2}$ over $\{ h_1^{\pm 1}, h_2^{\pm 1} \}$, see \er{e2}.

  Now we will  show that the group $G$ given by the presentation \er{e0} naturally embeds into the group $\GG_2$ given by  \er{e3}. Assume that $U_0$ is a cyclically reduced word  over $\A^{\pm 1}$ and $U_0 = 1$ in $\GG_2$. By Lemma \ref{vk}, there is a disk diagram $\D_0$ over \er{e3} such that  $\ph (\p \D_0) \equiv U_0$.  Without loss of generality, we may assume that $\D_0$ is reduced. Note that a boundary path of $\D_0$ contains no $h$-edges. If $\D_0$ contains no face of type F3 whose boundary path  has $h$-edges then, turning $h$-edges into 1-edges by relabeling,  we may assume that $\D_0$  contains no $h$-edges. Hence, we may suppose that
$\D_0$ is a disk diagram over the presentation  \er{e1}.
Then it follows from Lemmas \ref{lem1}, \ref{vk} that $U_0 = 1$ in $G$. Thus, if $U_0$ is not trivial in $G$,
then  $\D_0$ must contain a face of type F3 with $h$-edges.  Therefore, Lemma \ref{gr2} applies to $\D_0$ and yields the existence of a vertex $v_\Pi$, where $\Pi$ is a face of type F3
 with $h$-edges, whose degree $d$ in the graph $\Psi_h$  is positive and is at most $\max\{ 12(1-k), 12 \} = 12$ as $k = \chi( \D_0 ) =1$. It follows from the definition of the graph  $\Psi_h$ and Lemmas~\ref{gr1},~\ref{cntsd}  that there are $d \le 12$ maximal \cntsd s  $\Gamma_1, \dots, \Gamma_d$ between an $h$-section $q$ of $\Pi$ and some $h$-sections of $\D_0$ so that every edge of $q$ is contained in exactly one of the contiguity arcs $\Gamma_i \wedge q$, $i=1, \dots, d$.
Therefore, there is an index $i^*$ such that $| \Gamma_{i^*}  \wedge q | \ge \tfrac 1{12}
|q|$.  Since $\p \D_0$ contains no $h$-edges, it follows that  $ \Gamma_{i^*}$ is a \cntsd\ between $q$ and
$q'$, where  $q'$ is an  $h$-section of a face $\Pi'$.  Denote $q_\Pi :=
\Gamma_{i^*} \wedge q$ and   $q_{\Pi'} := \Gamma_{i^*} \wedge q'$.   Since
$\ph(q_\Pi ) \equiv \ph( q_{\Pi'} )^{-1}$ and $| q_{\Pi} | \ge   \tfrac 1{12} |q| >
\tfrac 4{M} |q|$ as $n \ge 2$ and $M = 24n \ge 48$, it follows from Lemma~\ref{lem2} that  $\ph(q) \equiv
\ph( q' )^{-1}$. Hence,  by the definition of relations in \er{e3} and by the definition of a \cntsd , we have that
$\ph(\p \Pi) \equiv \ph( \p \Pi' )^{- 1}$  and the faces $\Pi$, $\Pi'$ form a reducible pair. This contradiction to the fact that $\D_0$ is reduced proves that $U_0 \overset G = 1$  and,  therefore,  $G$  naturally embeds in $\GG_2$, as claimed. Let $\nu_2 : G \to \GG_2$ denote this embedding.
\medskip

Consider a quadratic equation $W=1$ over $G$ of length $\ell \le n$. We need to prove that the equation $W=1$ has a solution in  the group $G$ given by \er{e0} \ifff\  the equation $\nu_2(W)=1$   has a solution in  the group $\GG_2$ given by \er{e3}.

First assume that $W=1$ has a solution in $G$. By Lemma~\ref{lem1}, the equation $\nu_1(W)=1$  has a solution in the group $\GG_1$ given by \er{e1}. Since  $G$ naturally embeds in   $\GG_2$, it follows from the definition of presentations \er{e1}, \er{e3} that there is a homomorphism $\GG_1 \to \GG_2$ which is identical on $G$. Hence, we may conclude
that the equation $\nu_2(W)=1$  has a solution in the group $\GG_2$, as desired.

Conversely, suppose that the equation $\nu_2(W)=1$   has a solution in the group $\GG_2$. Our goal is to show that
 $W=1$ has a solution in $G$.  Let
 $$
 W \equiv t_1^{\e_1} U_1   t_2^{\e_2} U_2  \dots  t_\ell^{\e_\ell} U_\ell,
 $$
 where  $t_1, \dots, t_\ell \in \cup_{i=1}^\infty \X_i $, $\e_1, \dots, \e_\ell \in \{ \pm 1\}$, and $U_1, \dots, U_\ell$ are some reduced or empty words over $\A^{\pm 1}$. Since $\nu_2(W)=1$   has a solution in $\GG_2$, there are nonempty words
 $T_1, \dots, T_\ell$ over the alphabet $\cup_{i=1}^\infty \X_i^{\pm 1} \cup \A^{\pm 1}\cup \{ h_1^{\pm 1}, h_2^{\pm 1}, 1 \}$ such that
 $$
 T_1^{\e_1} U_1   T_2^{\e_2} U_2  \dots  T_\ell^{\e_\ell} U_\ell  \overset {\GG_2} = 1 .
 $$
 Note that we would use the letter $1$ for the trivial element of $\GG_2$.
 By Lemma~\ref{vk}, there is a disk diagram $\D$ over presentation \er{e3} such that
 $$
 \ph(\p \D) \equiv  T_1^{\e_1} U_1   T_2^{\e_2} U_2  \dots  T_\ell^{\e_\ell} U_\ell .
 $$

Since $W=1$ is a quadratic equation, there is  a permutation
$$
\tau : \{1, \dots, \ell \} \to \{1, \dots, \ell \}
$$
such that $\tau^2 =1$, $\tau(i) \ne i$ and $t_i = t_{\tau(i)}$ for every $i \in \{1, \dots, \ell \}$. Hence, we may assume that $T_i \equiv T_{\tau(i)}$ for every $i \in \{1, \dots, \ell \}$. Denote
$$
\p \D =  r_1^{\e_1} u_1   r_2^{\e_2} u_2  \dots r_\ell^{\e_\ell} u_\ell ,
$$
where $r_i, u_i$ are paths  of $\p \D^{\pm 1}$ such that $\ph(r_i) \equiv T_i$, $\ph(u_i) \equiv U_i$ for every $i =1, \dots, \ell$.
Now we construct a surface diagram $\wtl \D$ from $\D$ by attaching the path $r_i$ to $r_{\tau(i)}$ for every $i =1, \dots, \ell$.  Note that $\chi(\wtl \D) = 1- \tfrac \ell 2$ and $\wtl \D $ has $k'$ \cncm s in its boundary $\p \wtl \D $,  $1 \le k' \le \ell$. Thus,  $\wtl \D$ is a surface diagram of type $(1- \tfrac \ell 2, k')$.

Let $c_1, \dots c_{k'}$ be \cncm s of  $\p \wtl \D$.  Note that each $c_j$ is a product of some paths in the set
$\{ u_1^{\delta_1}, \dots, u_\ell^{\delta_\ell} \}$, where $\delta_1, \dots, \delta_\ell \in  \{ \pm 1 \}$, and each
 $u_j^{\delta_j}$ occurs in one of $c_1, \dots c_{k'}$ exactly once.  If $\wtl \D$ contains a reducible pair
 of faces, then we remove this pair by the surgery described above and  obtain a surface diagram $\wtl \D'$
 with unchanged boundary paths and $| \wtl \D'_3(2) |= |\wtl \D_3(2) |-2$, where $|\wtl \D_3(2) |$ is the number of faces of type F3 in $\D$. It is not difficult to check that there exists a disk diagram $\D'$ such that
 $$
 \p \D' =  (r'_1)^{\e_1} u'_1  ( r'_2)^{\e_2} u'_2  \dots (r'_\ell)^{\e_\ell} u'_\ell ,
 $$
where $r'_i, u'_i$ are paths  of $\p (\D')^{\pm 1}$ such that
$\ph(r'_i) \equiv \ph(r'_{\tau(i)}) \equiv T'_i$, $\ph(u'_i) \equiv \ph(u_i)$ for every $i =1, \dots, \ell$. Moreover, the diagram $\wtl \D'$ can be obtained from  $\D'$  in the same manner as $\wtl \D$ was obtained from  $\D$, in particular, $| \wtl \D'_3(2) |= |\D'_3(2) |$. Hence, by induction on  the number $| \D_3(2) |$ of faces of type F3 in $\D$,  we may assume that the surface diagram  $\wtl \D$ is reduced.

Suppose that  $\wtl \D$  contains no faces of type F3 with $h$-edges. Then $\D$ also has this property, hence we can turn $h$-edges of $\D$  into 1-edges by relabeling and obtain thereby a disk diagram $\bar \D$ from $\D$ with no $h$-edges. Such a diagram  $\bar \D$ could be regarded as a diagram over presentation \er{e1}. The existence of
such  $\bar \D$ over  \er{e1} means that the equation $\nu_1(W)=1$   has a solution in the group $\GG_1$ given by \er{e1}.
By Lemma~\ref{lem1}, the equation $W=1$ has a solution in  $G$, as required.

Hence, we may assume that $\D$ contains faces of type F3 with $h$-edges. Clearly,  $\wtl \D$  also has this property and we may consider the graph $\Psi_h = \Psi_h(\wtl \D)$ on  $\wtl \D$ as defined before.  Since  $\p \wtl \D$  contains no  $h$-edges, Lemma~\ref{gr2} applies to the graph $\Psi_h$ on  $\wtl \D$ and yields the existence of a vertex $v_\Pi$, where $\Pi$ is a face of $\wtl \D$, whose positive degree is at most
$$
\max\{ 12(1-k), 12 \}= \max\{ 6\ell, 12 \} =6\ell \le 6n
$$
as $\ell \ge 2$. As above, it follows from the definition of the graph  $\Psi_h$ and Lemmas~\ref{gr1},~\ref{cntsd}  that there are $d \le 6n$ maximal \cntsd s  $\Gamma_1, \dots, \Gamma_d$ between an $h$-section $q$ of $\Pi$ and some $h$-sections of $\wtl \D$ so that every edge of $q$ is contained in exactly one of the contiguity arcs $\Gamma_i \wedge q$, $i=1, \dots, d$.
Therefore, there is an index $i^*$ such that $| \Gamma_{i^*}  \wedge q | \ge \tfrac 1{6n}
|q|$.   Let $ \Gamma_{i^*}$ be a \cntsd\ between $q$ and
$q'$, where  $q'$ is an $h$-section of a face $\Pi'$.  Denote $q_\Pi :=
\Gamma_{i^*} \wedge q$ and   $q_{\Pi'} := \Gamma_{i^*} \wedge q'$.   Since
$\ph(q_\Pi ) \equiv \ph( q_{\Pi'} )^{-1}$ and $ | q_{\Pi} | \ge   \tfrac 1{6n} |q| =
\tfrac 4{M} |q|$ as $M = 24n$, it follows from Lemma~\ref{lem2} that  $\ph(q) \equiv
\ph( q' )^{-1}$. Hence,  by the definition of relations in \er{e3} and by the definition of a \cntsd , we have that
$\ph(\p \Pi) \equiv \ph( \p \Pi' )^{-1}$  and the faces $\Pi$, $\Pi'$ form a reducible pair in $\wtl \D$.
This contradiction to the fact that $\wtl \D$ is reduced  proves that it is impossible that $\D$ contains faces of type F3 with $h$-edges.
Hence, the equation $W=1$ has a solution in $G$, as desired.

Thus, the group $\GG_2$ has all of the required properties of the group $H$ of the statement of Theorem~\ref{thm1} and the proof of Theorem~\ref{thm1} is complete.
\qed
\medskip

{\em Proof of Theorem \ref{thm2}.}  (a) \ Let
$
W_1 =1 ,  W_2 =1, \dots
$
be the enumeration, fixed in \eqref{enum0}, of all quadratic equations over $G$ such that, for every $i \ge 1$,   $| W_i|_\X \le n $ and $W_i =1$ has a solution in $G$. Recall that the enumeration $\cup_{i=1}^{\infty} \X_i= \{x_1,x_2,\dots\}$ has the property that if $x_j \in \X_k$, $x_{j'} \in \X_{k'}$ and $k <k'$ then  $j <j'$. This property implies, for every $W_i=1$, that if $x_{k_1},\dots,x_{k_{\ell}}$ are the letters of $\cup_{i=1}^{\infty} \X_i$ that appear in $W_i(\X_i)^{\pm 1}$, then $k_1,\dots, k_{\ell}\leq ni$. Thus, in view of the  relations $x_j^{-1} V_{2j} = 1$ of the
presentation \er{e3}, it follows that $(V_{2k_1}, \dots, V_{2k_\ell} )$ is a solution tuple to the equation $\nu_2(W_i)=1$ over $\GG_2$. Since  $|V_k| \le (M(k+1) +1)M$ and $\ell \le n$, we further obtain that
\begin{equation*}
\sum_{j'=1}^\ell | V_{2k_{j'}} | \le nM( M(2ni+1) +1 ) \le 3n^2M^2 i = C n^4 i ,
\end{equation*}
where $C = 3\cdot 24^2$ as $M = 24n$.
\smallskip

(b) Since the presentation \er{e0} of $G$ is recursively enumerable, it follows that the set of all words
$U$ over $\A^{\pm 1}$  such that $U \overset G = 1$  is also  recursively enumerable. More generally, we can analogously obtain that  all quadratic equations  $W =1$ over $G$ of length $\le n$ that have solutions in $G$ can be
recursively enumerated. The last observation means that we can create a recursive enumeration \eqref{enum0}. Now we can  use constructions of \er{e3}, \er{prh} and
see that defining relations of the presentation \er{prh} for $\GG_2$ can
be recursively enumerated as well.
\smallskip

(c) The existence of an algorithm that detects whether a quadratic equation $W=1$ over $G$ of length $\le n$ has a solution enables us to effectively write down all quadratic equations  $W =1$ over $G$ of length $\le n$ that have solutions in $G$. Hence, we can effectively  create an enumeration  \eqref{enum0} and, using   constructions of \er{e3}, \er{prh}, write down
 all relations of the form  $\wht  W_1 =1,      \wht    W_2 =1, \dots$ in the presentation  \er{prh}.
 Since the presentation \er{e0} of $G$ is decidable, we can also effectively
write down all relations  of the form  $\wht R_1=1,  \wht R_2=1, \dots$ in the presentation  \er{prh}. Hence, the presentation  \er{prh} is decidable. Since the map $a_i \to V_{2i+1}$, $i=1,2, \dots$, extends to the embedding  $\mu_n : G \to H $ and the set of defining relations of  presentation  \er{prh} is recursive,  we see that the embedding $\mu_n : G \to H$ can be effectively constructed.
  Theorem~\ref{thm2} is proven.  \qed

\end{document}